\documentclass[oneside]{amsart}
\usepackage{amssymb}

\usepackage{graphicx}

\usepackage[cmtip,all]{xy}
\usepackage[latin1]{inputenc}
\usepackage{amssymb,amsmath}
\usepackage{verbatim}
\usepackage{color}
\usepackage{enumerate}
\usepackage{url}
\usepackage{geometry}
\usepackage{hyperref}
\usepackage{booktabs} 

\usepackage{array}

\newtheorem{theorem}{Theorem}[section]
\newtheorem*{theorem*}{Theorem}
\newtheorem{lemma}[theorem]{Lemma}
\newtheorem{proposition}[theorem]{Proposition}

\newtheorem{corollary}[theorem]{Corollary}

\theoremstyle{definition}

\newtheorem{assumption}[theorem]{Assumption}

\theoremstyle{remark}
\newtheorem{remark}[theorem]{Remark}

\numberwithin{equation}{section}

\newcommand\Z{\ensuremath{\mathbb Z}}\newcommand\A{\ensuremath{\mathbb A}}
\newcommand\Q{\ensuremath{\mathbb Q}}
\newcommand\C{\ensuremath{\mathbb C}}

\newcommand\Aut{\operatorname{Aut}}

\newcommand\cycl{\operatorname{cycl}}

\newcommand\Frob{\operatorname{Frob}}
\newcommand\Gal{\operatorname{Gal}}
\newcommand\GL{\operatorname{GL}}
\newcommand\Hom{\operatorname{Hom}}

\newcommand\Ind{\operatorname{Ind}}

\newcommand{\fp}{{\mathfrak{p}}}
\newcommand{\fP}{{\mathfrak{P}}}
\newcommand{\fc}{{\mathfrak{c}}}

\newcommand{\cE}{{\mathcal{E}}}

\renewcommand{\L}{\mathcal{L}}
\newcommand{\tto}[1]{%
\ifthenelse{\equal{#1}{}}{\to}{\stackrel{#1}{\to}}}

\def\cO{{\mathcal O}}

\def\p{\mathfrak p}
\newcommand{\comp}{\begin{picture}(6,5)(-3,-2)\put(0,1){\circle{2}} \end{picture}}\def\circ{\comp}

\def\res{\operatorname{res}}

\newcommand{\ra}{\rightarrow}

\newcommand{\lra}{\longrightarrow}

\def\XXint#1#2#3{{\setbox0=\hbox{$#1{#2#3}{\int}$}
\vcenter{\hbox{$#2#3$}}\kern-.5\wd0}}

\newcommand{\bff}{{\bf f}}
\newcommand{\bfg}{{\bf g}}
\newcommand{\bfh}{{\bf h}}

\newcommand\coh{\operatorname{H}}

\newcommand\norm{\operatorname{N}}

\newcommand{\Sel}{\operatorname{Sel}}

\usepackage{listings}
\renewcommand{\hom}{\mathrm{Hom}}

\begin{document}

\title[Triple product $p$-adic $L$-functions and non-crystalline classes]{Special values of triple-product $p$-adic L-functions and non-crystalline diagonal classes}

\author{Francesca Gatti, Xavier Guitart, Marc Masdeu, and Victor Rotger}
\address{}
\email{}
\thanks{This project has received funding from the European Research Council (ERC) under the European Union's Horizon 2020 research and innovation programme (grant agreement No 682152). The authors were also partially supported by projects MTM2015-66716-P and  MTM2015-63829-P}

\date{\today}

\dedicatory{}

\begin{abstract}
The main purpose of this note is to understand the arithmetic encoded in the special value of the $p$-adic $L$-function $\L_p^g(\bff,\bfg,\bfh)$ associated to a triple of modular forms $(f,g,h)$ of weights $(2,1,1)$, in the case where the classical $L$-function $L(f\otimes g\otimes h,s)$ --which typically has sign $+1$-- does not vanish at its central critical point $s=1$. When $f$ corresponds to an elliptic curve $E/\Q$ and the classical $L$-function vanishes,  the Elliptic Stark Conjecture of Darmon--Lauder--Rotger predicts that $\L_p^g(\bff,\bfg,\bfh)(2,1,1)$ is either $0$ (when the order of vanishing of the complex $L$-function is $>2$) or related to logarithms of global points on $E$ and a certain Gross--Stark unit associated to $g$. We complete the picture proposed by the Elliptic Stark Conjecture by providing a formula for the value $\L_p^g(\bff,\bfg,\bfh)(2,1,1)$ in the case where $L(f\otimes g\otimes h,1)\neq 0$.
\end{abstract}
\maketitle 

\tableofcontents

\vspace{0.15cm}

\section{Introduction}\label{sec: Introduction}

Let $E$ be an elliptic curve defined over $\Q$ and let $f\in S_2(N_f)$ be the newform attached to $E$. Let 
\begin{equation*}
g\in S_1(N_g,\chi)_L, \qquad h\in S_1(N_h,\bar\chi)_L
\end{equation*}
be two cuspforms of weight one, inverse nebentype characters and with Fourier coefficients contained in a number field $L$. Let $\rho_g$ and $\rho_h$ be the Artin representations attached to $g$ and $h$. The tensor product $\rho_g\otimes\rho_h$ is a self-dual Artin representation of dimension $4$ of the form $$\rho:=\rho_g\otimes\rho_h:\Gal(H/\Q)\hookrightarrow\Aut(V_g\otimes V_h)\cong\GL_4(L),$$  where $H/\Q$ is a finite extension.

In this setting, the complex $L$-function $L(E\otimes\rho,s)$ attached to the (Tate module $V_p(E)$ of the) elliptic curve $E$ twisted by the Artin representation $\rho$ coincides with the Garrett--Rankin $L$-function $L(f\otimes g\otimes h,s)$ attached to the triple $(f,g,h)$ of modular forms. 
By multiplying this $L$-function by an appropriate archimedean factor $L_\infty(f\otimes g\otimes h, s)$ one obtains an entire function $\Lambda(f\otimes g\otimes h,s)$ which satisfies a functional equation of the form
\begin{equation}\label{eq: functional equation triple product (2,1,1)}
\Lambda(f\otimes g\otimes h, s)=\epsilon\cdot\Lambda(f\otimes g\otimes h, 2-s),
\end{equation}
where $\epsilon\in\{ \pm1 \}$.
Moreover, $L_\infty(f\otimes g\otimes h,s)$ does not have zeros nor poles at $s=1$.

Denote $N_g$ and  $N_h$ the level of $g$ and $h$ respectively. The sign can be written as a product of local factors $\epsilon=\prod_{v}\epsilon_v$ where $v$ runs over the places of $\Q$, and $\epsilon_v=+1$ if $v$ is a finite prime which does not divide $\operatorname{lcm}(N_f,N_g,N_h)$ or if $v=\infty$. 
We will work under the following assumption
\begin{assumption}\label{ass: local signs}
$	\epsilon_v=+1  \quad \text{ for all } v.$
\end{assumption}
Assumption \ref{ass: local signs} holds most of the time: this is the case for instance if the greatest common divisor of the levels of $f,g$ and $h$ is $1$.

Fix an odd prime number $p$ such that 
\begin{equation*}
p\nmid N_fN_gN_h,
\end{equation*} 
and denote by $\alpha_g,\beta_g$ the eigenvalues for the action of the Frobenius element at $p$ acting on $V_g$. We use the analogous notation for $h$, and we assume
\begin{equation*}
\alpha_g\neq\beta_g, \text{ and } \ \alpha_h\neq\beta_h.
\end{equation*}
Fix once and for all completions $H_p, L_p$ of the number fields $H, L$ at primes above $p$.

Choose an ordinary $p$-stabilisation of $g$, namely $g_\alpha(z):=g(z)-\beta_gg(pz)$ and define analogously $h_\alpha$. Let $\bff,\bfg,\bfh$ be Hida families passing through the unique ordinary $p$-stabilisation of $f$ and  $g_\alpha$ and $h_\alpha$  respectively, and consider the Garret--Hida $p$-adic $L$-function $$\L_p^g(\bff, \bfg,\bfh) $$ of \cite{DR1} associated to the specific choice of test vectors $(\breve{\bff},\breve{\bfg},\breve{\bfh})$ of \cite[Chap. 3]{Hsieh}. This $p$-adic $L$-function interpolates the square-roots of the central values of the classical $L$-function $L(\breve{f}_k\otimes \breve{g}_\ell\otimes \breve{h}_m,s)$ attached to the specializations of the Hida families at classical points of weights $k,\ell,m$  with $k,\ell,m\geq 2$ and $\ell\geq k+m$. 
Notice that the point $(2,1,1)$, which corresponds to our triple of modular forms $(f,g,h)$, lies outside the region of classical interpolation for $\L_p^g(\bff, \bfg,\bfh)$.
We are interested in studying the value 
\begin{equation*}
\L_p^g(\bff, \bfg,\bfh)(2,1,1)
\end{equation*}
under the following assumption:
\begin{assumption}\label{ass: rank 00}
	$L(E\otimes\rho,1)\neq0 \text{ and } \Sel_p(E\otimes\rho)=0.$
\end{assumption}

Here $\Sel_p(E\otimes\rho)$ denotes the Bloch--Kato Selmer group attached to the representation $$V:=V_p(E)\otimes V_g\otimes V_h.$$

Under Assumption \ref{ass: local signs}, the sign $\epsilon$ of the functional equation \eqref{eq: functional equation triple product (2,1,1)}  is $+1$, and thus the order of vanishing of $L(E\otimes \rho,s)$ at $s=1$ is even. One hence expects that $L(E\otimes\rho,1)$ is generically nonzero. If this $L$-value is nonzero, by \cite{DR2} we know that the $\rho$-isotypical component $E(H)^{\rho}:=\hom_{G_\Q}(V_g\otimes V_h,E(H)\otimes L)$ of the Mordell-Weil group $E(H)$ is trivial. By the Shafarefich--Tate conjecture one also expects the Selmer group $\Sel_p(E\otimes\rho)$ to be trivial, although this conjecture is widely open. 
It is also worth noting that the value $\L_p^g(\bff, \bfg,\bfh)(2,1,1)$ in the setting in which the complex $L$-function $L(E\otimes\rho,s)$ vanishes at $s=1$ has been analyzed in \cite{DLR}, where the authors give a conjectural formula for this $p$-adic value as a $2\times 2$-regulator of $p$-adic logarithms of global points.

Under our running assumption \ref{ass: rank 00} one can not expect a similar formula for the above $p$-adic $L$-value, as no global points are naturally present in this scenario.
The {\em main result of this paper} consists in an explicit formula for the value $\L_p^g(\bff, \bfg,\bfh)(2,1,1)$  which involves the algebraic part of the classical $L$-value $L(E\otimes\rho,1)$ and the {\em logarithm} of a canonical non-crystalline class along a certain {\em crystalline direction}. 

In \S\ref{sec: Selmer groups} we recall the basic definitions on Selmer groups and we give a precise description of the relaxed $p$-Selmer group $\Sel_{(p)}(E\otimes\rho)$ under Assumption \ref{ass: rank 00}. More precisely,  the projection to the singular quotient gives an isomorphism
\begin{equation}\label{eq: iso for relaxed selmer}
\partial_p:\Sel_{(p)}(E\otimes\rho)\overset{\cong}{\longrightarrow}\coh^1_s(\Q_p,V).
\end{equation}
Let $V_g^\alpha, V_g^\beta$, with basis $v_g^\alpha, v_g^\beta$ respectively, be the eigenspaces of $V_g$ for the action of $\Frob_p$ with eigenvalues $\alpha_g, \beta_g$, and use the analogous notation for $V_h$. The $G_{\Q_p}$-representation $V$ decomposes as a direct sum as
\begin{equation*}
V= V^{\alpha\alpha}\oplus V^{\alpha\beta}\oplus V^{\beta\alpha}\oplus V^{\beta\beta},
\end{equation*} 
where $V^{\alpha\alpha}:=V_pE\otimes V_g^\alpha\otimes V_h^\alpha$ and similarly for the other pieces.
It induces the decomposition 
\begin{equation}\label{eq: decomp H1s}
\coh^1_s(\Q_p,V)= \coh^1_s(\Q_p,V^{\alpha\alpha})\oplus \coh^1_s(\Q_p,V^{\alpha\beta})\oplus \coh^1_s(\Q_p,V^{\beta\alpha})\oplus \coh^1_s(\Q_p,V^{\beta\beta}),
\end{equation} 
and the Bloch--Kato dual exponential gives isomorphisms
\begin{equation*}
\exp^*_{\alpha\alpha}:\coh^1_s(\Q_p,V^{\alpha\alpha})\overset{\cong}{\longrightarrow}L_p
\end{equation*}
and similarly for the other pieces of the decomposition \eqref{eq: decomp H1s}.
Combining it with \eqref{eq: iso for relaxed selmer}, we get a basis 
\begin{equation*}
\xi^{\alpha\alpha}, 	\xi^{\alpha\beta}, 	\xi^{\beta\alpha}, 	\xi^{\beta\beta}
\end{equation*}
for $\Sel_{(p)}(E\otimes\rho)$ characterised by the fact that
\begin{equation*}
\partial_p\xi^{\alpha\alpha}\in\coh_s^1(\Q_p,V^{\alpha\alpha}) \text{ and } \exp^*_{\alpha\alpha}\partial_p\xi^{\alpha\alpha}=1,
\end{equation*}
and similarly for $	\xi^{\alpha\beta}, 	\xi^{\beta\alpha}, 	\xi^{\beta\beta}$.

The $G_{\Q_p}$-cohomology of $V$ and its submodule of crystalline classes $\coh^1_f(\Q_p,V)\subseteq\coh^1(\Q_p,V)$ also have decompositions analogous to \eqref{eq: decomp H1s}. Moreover, if 
\begin{equation*}
\pi_{\alpha\beta}:\coh^1(\Q_p,V)\longrightarrow\coh^1(\Q_p,V^{\alpha\beta})
\end{equation*}
denotes the projection, then $\pi_{\alpha\beta}\xi^{\beta\beta}$ lies in $\coh^1_f(\Q_p,V^{\alpha\beta})$.
Finally, we can write
\begin{equation*}
\pi_{\alpha\beta}\xi^{\beta\beta}=R_{\beta\alpha}\otimes v_g^\alpha\otimes v_h^\beta\in(E(H_p)\otimes V_{g}^\alpha\otimes V_h^\beta)^{G_{\Q_p}}\cong\coh^1_f(\Q_p,V^{\alpha\beta})
\end{equation*}
where $R_{\beta\alpha}\in E(H_p)$ is a local point on which $\Frob_p$ acts as multiplication by $\beta_g\alpha_h$.

We can finally state the main result of the paper.

\begin{theorem*}[cf Theorem~\ref{thm: main}]
	Under Assumptions \ref{ass: p doesnt divide the levels} and \ref{ass: rank 00},
	\begin{equation}\label{eq: main thm intro}
	\L_p^g(\bff, \bfg,\bfh)(2,1,1) = \dfrac{A \cdot\cE  }{\pi\langle f, f\rangle} \times \dfrac{\log_p(R_{\beta\alpha})}{\L_{g_\alpha}}\times\sqrt{L(E\otimes\rho,1)},
	\end{equation}
	where $A\in\Q^\times$ is an explicit number, $\cE\in L_p$ is a product of Euler factors, $\langle f, f\rangle$ denotes the Petersson norm of $f$, $\L_{g_\alpha}\in H_p$ is an element on which $\Frob_p$ acts as multiplication by $\frac{\beta_g}{\alpha_g}$ and which only depends on $g_\alpha$, and $\log_p:E(H_p)\rightarrow H_p$ denotes the $p$-adic logarithm.
	
\end{theorem*}

We refer to Theorem \ref{thm: main} for a more precise statement of the result and of the objects appearing in \eqref{eq: main thm intro}. In particular, the element $\L_{g_\alpha}$ is expected to be related to a so-called Gross--Stark unit attached to $g_\alpha$, as conjectured in \cite[Conjecture 2.1]{DR2b}.

Under the additional assumption that $g$ is not the theta series of a Hecke character of a real quadratic field in which $p$ splits, the value $\L_p^g(\bff, \bfg,\bfh)(2,1,1)$ can be recast in a more explicit way in terms of \textit{$p$-adic iterated integrals}, as explained in the introduction of \cite{DLR}. The numerical computations we offer in \S \ref{sec: Numerical computations} are obtained by calculating such integrals, where a key input are Lauder's algorithms \cite{Lauder} for the computation of overconvergent projections.

As an application of the main result, in Section \ref{sec: p split in K} we explore the situation where $g$ and $h$ are theta series of the same imaginary quadratic field in which $p$ splits. The following theorem is stated as Theorem \ref{thm: CM case} in the text.

\begin{theorem*}[cf Theorem~\ref{thm: CM case}]
Let $K$ be an imaginary quadratic field in which $p$ is split, and let $\psi_g$ (resp. $\psi_h$) be a finite order Hecke character of $K$ of conductor $\fc_g$ (resp. of conductor $\fc_h)$.  Denote by $g$ and $h$ the theta series attached to $\psi_g$ and $\psi_h$, respectively. Suppose that $\gcd(N_f,\fc_g,\fc_h)=1$ and that the Nebentype characters of $g$ and $h$ are inverses to each other. If $L(E,\rho_g\otimes\rho_h,1)\neq 0$ then $\L_p^g(\bff,\bfg,\bfh)(2,1,1)=0$.
\end{theorem*}

\section{The Selmer group of $f\otimes g\otimes h$}\label{sec: Selmer groups}
We begin this section by collecting some standard facts on Selmer groups of $p$-adic Galois representations that we will use. Then we introduce the Galois representation attached to the triple of modular forms $f$, $g$, and $h$ of weights $2$, $1$, $1$, and we study the corresponding Selmer groups. In particular, the structure of the relaxed Selmer group will be key in proving the main theorem of Section \ref{sec: main thm}.

\subsection{Selmer groups}
Let $V$ be a $\Q_p[G_\Q]$-module and let $B_\mathrm{cris}$ be Fontaine's $p$-adic crystalline period ring.  For each prime number $\ell$, denote 
\begin{equation}\label{eq: selmer structure}
	\coh_f^1(\Q_\ell,V):=\begin{cases} \coh^1_\mathrm{ur}(\Q_\ell,V):=\coh^1(\Q_\ell^\mathrm{ur}/\Q_\ell,V^{I_\ell}) & \ell\neq p\\ \ker\Big(\coh^1(\Q_p,V)\rightarrow\coh^1(\Q_p,V\otimes_{\Q_p} B_\mathrm{cris})   \Big) & \ell=p,	\end{cases}
\end{equation}
and 
\begin{equation*}
	\coh^1_s(\Q_\ell,V):=\coh^1(\Q_\ell,V)/\coh^1_f(\Q_\ell,V).
\end{equation*}
The Bloch--Kato Selmer group of $V$ is $$\Sel_p(\Q,V):=\{x\in\coh^1(\Q,V)\mid \res_\ell(x)\in\coh_f^1(\Q_\ell,V) \text{ for all }\ell \},$$
where $\res_\ell\colon H^1(\Q,V)\ra H^1(\Q_\ell,V)$ denotes the restriction map in Galois cohomology.

For each prime $\ell$, we denote by $\partial_\ell$ the composition
\begin{equation*}
	\partial_\ell: \coh^1(\Q,V)\overset{\res_\ell}{\longrightarrow}\coh^1(\Q_\ell,V)\longrightarrow \coh^1_s(\Q_\ell,V),
\end{equation*}
where the second map is the natural quotient map.

The relaxed Selmer group is defined as $$\Sel_{(p)}(\Q,V):=\{ x\in\coh^1(\Q,V)\mid \res_\ell(x)\in\coh_f^1(\Q_\ell,V) \ \text{ for all }\   \ell\neq p \}\supseteq\Sel_p(\Q,V).$$

Let $V^*:=\Hom_{\Q_p}(V,\Q_p(1))$ be the 
Kummer dual of $V$. One can define a Selmer group $\Sel_{p,*}(\Q,V^*)$ for $V^*$ which is dual to (\ref{eq: selmer structure}) with respect to the local Tate pairings
\begin{align}\label{eq: tate pairing}
\langle \ , \ \rangle_\ell:\coh^1(\Q_\ell,V)\times\coh^1(\Q_\ell,V^*)\longrightarrow\Q_p.
\end{align}
For each $\ell$, define $\coh_{f,*}^1(\Q_\ell,V^*)$ to be the orthogonal complement of $\coh_{f}^1(\Q_\ell,V)$ with respect to \eqref{eq: tate pairing}; the Selmer group attached to $V^*$ is then 
\begin{align*}
\Sel_{p,*}(\Q,V^*):=\{x\in\coh^1(\Q,V^*)\mid \res_\ell(x)\in\coh_{f,*}^1(\Q_\ell,V^*) \ \text{ for all } \ell  \}.  
\end{align*}
Finally, the strict Selmer group of $V^*$ is the subspace of $\Sel_{p,*}(\Q,V^*)$ defined as
\begin{align*}
\Sel_{[p],*}(\Q,V^*):=\{ x\in \coh^1(\Q,V^*)\mid \res_\ell(x)\in \coh_{f,*}^1(\Q_\ell,V^*)  \text{ for all } \ell \text{ and } \res_p(x)=0 \}.  
\end{align*}

By Poitou--Tate duality (see, for example, \cite[Theorem 2.3.4]{MR04}) there is an exact sequence
\begin{equation}\label{eq: exact sequence of Selmer groups}
0\rightarrow \Sel_p(\Q,V)\rightarrow\Sel_{(p)}(\Q,V)\rightarrow\coh_s^1(\Q_p,V)\rightarrow\Sel_{p,*}(\Q,V^*)^\lor\rightarrow\Sel_{[p],*}(\Q,V^*)^\lor,
\end{equation}
where ${}^\lor$ stands for the $\Q_p$-dual.

\subsection{Representations attached to modular forms} In this section we review the main features of the representations, both $p$-adic and $\Lambda$-adic, attached to modular forms in the lines of \cite[\S2]{DR2b}, which the reader can consult for more details.

Let $f\in S_2(N_f)$ be a weight two normalized eigenform of level $N_f$, trivial nebentype character and rational Fourier coefficients $a_n(f)$. Denote by $E$ the elliptic curve over $\Q$ of conductor $N_f$ associated to $f$ by the Eichler--Shimura construction. 

Let also
\begin{equation*}
	g\in S_1(N_g,\chi) \ \text{ and } \ h\in S_1(N_h,\bar{\chi})
\end{equation*}
be two normalized newforms of weight one, levels $N_g$ and $N_h$, and nebentype characters $\chi$ and $ \bar{\chi}$ respectively. Denote by $K_g$ and $K_h$ their fields of Fourier coefficients, and put $L:=K_g\cdot K_h$ the compositum of these fields.  

From now on, we fix a rational prime $p$, and we assume the following hypothesis.
 \begin{assumption}\label{ass: p doesnt divide the levels}
 The prime $p$ does not divide $ N_f N_g N_h.$
 \end{assumption} 
Since we will be interested in putting $f$ in a Hida family, we assume moreover that $f$ is ordinary at $p$; that is to say, that $p\nmid a_p(f)$.
 
We denote the $2$-dimensional $p$-adic representations attached to $f$ by $V_f$. Since $f$ corresponds to the curve $E$, the representation $V_f$ is given by the rational Tate module  $V_p(E)=T_p(E)\otimes\Q_p$. Denote by $\alpha_f,\beta_f$ the roots of the Hecke polynomial $X^2-a_p(f)X + p$. Since $f$ is ordinary at $p$, one of these roots, say $\alpha_f$, is a $p$-adic unit; also, the restriction of $V_f$ to a decomposition group $G_{\Q_p}\subset G_\Q$ admits a filtration of $\Q_p[G_{\Q_p}]$-modules  
$$0\rightarrow V^+_f\longrightarrow V_f\longrightarrow V^-_f\rightarrow 0$$
with the following properties:
\begin{enumerate}
\item $\dim_{\Q_p}V^+_f= \dim_{\Q_p}V^-_f=1$;
\item the group $G_{\Q_p}$ acts on the quotient $V_f^-$ via $\psi_f$, where  $\psi_f:G_{\Q_p}\rightarrow\Z_p^\times$ is the unramified character that maps an arithmetic Frobenius $\Frob_p$ to $\alpha_f$.
\item the group $G_{\Q_p}$ acts on  $V_f^+$ via the character $\chi_{\cycl}\psi_f^{-1}$ (here $\chi_{\mathrm{cycl}}$ is the $p$-adic cyclotomic character).
\end{enumerate}

There are Artin representations associated to $g$ and $h$. Without loss of generality we can assume that they are defined over $L$, and that they factor through the same finite extension $H$ of $\Q$. That is to say, they are of the form
\begin{align*}
  \rho_g\colon \Gal(H/\Q)\lra \Aut(V_g^{\circ})\cong \GL_2(L),\ \ \   \rho_h\colon \Gal(H/\Q)\lra \Aut(V_h^{\circ})\cong \GL_2(L)
\end{align*}
for certain $2$-dimensional $L$-vector spaces $V_g^{\circ}$ and $V_h^{\circ}$.

Fix once and for all a prime $\fp$ of $H$ and a prime $\fP$ of $L$ above $p$. Denote the corresponding completions by $H_p:=H_\fp$ and $L_p:=L_\fP$. There are also $p$-adic Galois representations associated to $g$ and $h$, that we will denote by $V_g$ and $V_h$. There are non-canonical isomorphisms
\begin{align}\label{eq: def of Vg and Vh}
 j_g:V^{\circ}_g\otimes_L L_p\overset{\cong}{\longrightarrow} V_g \ \text{ and } \  j_h: V^{\circ}_h\otimes_L L_p\overset{\cong}{\longrightarrow} V_h.
\end{align}

Since $ p\nmid N_g N_h$ the representations $V_g$ and $V_h$ are unramified at $p$. We assume from now on that $\Frob_p$ acts on $V_g$ and $V_h$ with distinct eigenvalues. Let $\alpha_g$, $ \beta_g$ be the eigenvalues for the action of $\Frob_p$ on $V_g$ and let  $V_g^\alpha, V_g^\beta$ be the corresponding eigenspaces. We will use the analogous notations $\alpha_h$, $\beta_h$, $V_h^\alpha$, and $V_h^\beta$ for $h$.

Denote by $g_\alpha$ the $p$-stabilisation of $g$ such that $U_p(g_\alpha)=\alpha_gg_\alpha$. The theory of Hida families ensures the existence of a Hida family $\bfg$ passing through $g_\alpha$. This can be regarded as a power series $\bfg\in \Lambda_\bfg[[q]]$, where $\Lambda_\bfg$ is a finite flat extension of the Iwasawa algebra $\Lambda:=\Z_p[[T]]$, with the property that, if we denote by $y_g\colon \Lambda_\bfg\ra L_p$ the weight corresponding to $g$, then $y_g(\bfg)=g_\alpha$. There is a locally free $\Lambda_\bfg$-module $V_{\bfg}$ and a $\Lambda$-adic representation $\rho_\bfg:G_\Q\rightarrow\GL(V_\bfg)\cong \GL_2(\Lambda_\bfg)$ that interpolates the $p$-adic representations associated to the specializations of $\bfg$.  As a $G_{\Q_p}$-representation, $V_\bfg$ is equipped with a filtration of $\Lambda_\bfg[G_{\Q_p}]$-modules
\begin{equation}\label{eq: filtration of lambda-adic rep}
0\rightarrow V_\bfg^+\longrightarrow V_\bfg\longrightarrow V_\bfg^-\rightarrow0,
\end{equation}
where $V_\bfg^+$ and $V_\bfg^-$ are locally free of rank one and the action of $G_{\Q_p}$ on $V_\bfg^-$ is unramified, with $\Frob_p$ acting as multiplication by the $p$-th Fourier coefficient of $\bfg$. There is a perfect Galois equivariant pairing  
\begin{align}\label{eq: lambda-adic pairing}
\langle \ , \ \rangle:V_\bfg^-\times V_\bfg^+\longrightarrow\Lambda_\bfg(\det(\rho_\bfg)).
\end{align}
For a crystalline $\Q_p[G_{\Q_p}]$-module $W$, denote $D(W):=(W\otimes B_\mathrm{cris})^{G_{\Q_p}}$. Recall that, if $W$ is unramified, then 
\begin{equation*}
 D(W)\cong(W\otimes\hat{\Q}_p^{\mathrm{ur}})^{G_{\Q_p}},
\end{equation*}
where $\hat{\Q}_p^{\mathrm{ur}}$ is the $p$-adic completion of the maximal unramified extension of $\Q_p$.  Denote by $\omega_\bfg\in D(V_\bfg^-)$ the canonical period associated to $\bfg$ constructed by  Ohta \cite{Ohta}.

By specializing via $y_g$, we obtain the $L_p$-vector space $y_g(V_\bfg):=V_\bfg\otimes_{\Lambda_\bfg,y_g}L_p$, which can be identified with $V_g$. Using the functoriality of $D$ and the identification $y_g(V_\bfg^+)=V_g^\beta, \ y_g(V_\bfg^-)=V_g^\alpha$ we obtain a pairing 
\begin{equation}\label{eq: pairing Vg+}
\langle \ , \ \rangle:D(V_g^\alpha)\times D(V_g^\beta)\longrightarrow D(L_p(\chi))=(H_p\otimes L_p(\chi))^{G_{\Q_p}}.
\end{equation}

Define 
\begin{align*}
\omega_g:=y_g(\omega_\bfg)\in D(V_g^\alpha)
\end{align*}
and let 
\begin{align*}
\eta_g\in D(V_g^\beta)
\end{align*}
be the element characterized by the equality
\begin{equation}\label{eq: pairing between etas and omegas}
\langle\omega_g,\eta_g\rangle= \mathfrak{g}(\chi)\otimes 1\in D(L_p(\chi)),
\end{equation} 
where $\mathfrak{g}(\chi)$ denotes the Gauss sum of $\chi$ viewed as an element of $H_p$. 
We define similarly $\omega_h\in D(V_h^\alpha)$ and  $ \eta_h\in D(V_h^\beta)$.

Using the isomorphisms \eqref{eq: def of Vg and Vh} we can define an $L$ structure on $V_g$ by  $V_g^L:=j_g(V^{\circ}_g)$. Let $v_g^{\alpha}$ (resp. $v_g^{\beta}$) be an $L$-basis of $V_g^L\cap V_g^{\alpha}$ (resp. of $V_g^{\beta}$). Define  
\begin{align*}
\Omega_g\in H_p^{1/\alpha_g}, \ \Theta_g\in H_p^{1/\beta_g}  
\end{align*}
to be the elements such that 
 \begin{equation}\label{eq: definition of Omega and Theta}
 	\Omega_g\otimes v_g^{\alpha}=\omega_g\in D(V_g^\alpha), \ \ \Theta_g\otimes v_g^\beta=\eta_g\in D(V_g^\beta).
 \end{equation}
 
  Let
 \begin{align*}
   V:=V_{f}\otimes V_{gh}
 \end{align*}
be the $p$-adic representation given by the tensor product $V_f\otimes V_g\otimes V_h$. Since the product of the nebentype characters of $f$, $g$, and $h$ is trivial we have that  $V^*\cong V$. We next study the structure of several Selmer groups associated to $V$.

\subsection{Selmer groups of $V$}

Put $V_{gh}^{\circ}:=V_g^{\circ}\otimes_L V_h^{\circ}$ and denote by $\rho$ the representation afforded by this space:
\begin{align*}
  \rho\colon \Gal(H/\Q)\lra \Aut(V_{gh}^{\circ}).
\end{align*}
Put $E(H)_L:=E(H)\otimes_\Z L$ and denote by $E(H)^{\rho}$ the $\rho$-isotypical component of the Mordell--Weil group:
\begin{equation*}
E(H)^{\rho}:=\hom_{\Gal(H/\Q)}(V_{gh}^{\circ},E(H)_L).
\end{equation*}

\begin{lemma}\label{lem: first lemma on isomorphisms cohomology triple product}
	There are isomorphisms
	\begin{align}
	\coh^1(\Q,V)\cong(\coh^1(H,V_f)\otimes V_{gh})^{\Gal(H/\Q)}\cong\hom_{\Gal(H/\Q)}(V_{gh},\coh^1(H,V_f));\label{eq: iso cohomology triple product1}\\
	\coh^1(\Q_p,V)\cong(\coh^1(H_ p,V_f)\otimes V_{gh})^{\Gal(H_ p/\Q_p)}\cong\hom_{\Gal(H_ p/\Q_p)}(V_{gh},\coh^1(H_ p,V_f))\label{eq: iso cohomology triple product2}. 
	\end{align}
\end{lemma}
\begin{proof}
We prove only \eqref{eq: iso cohomology triple product2}, and \eqref{eq: iso cohomology triple product1} is proven similarly. By the inflation-restriction  exact sequence we have the exact sequence
	\begin{equation*}
0\rightarrow \coh^1(\Gal(H_ p/\Q_p), V^{G_{H_ p}})\rightarrow \coh^1(\Q_p,V)\rightarrow \coh^1(H_ p, V)^{\Gal(H_ p/\Q_p)}\rightarrow\coh^2(\Gal(H_ p/\Q_p), V^{G_{H_ p}}).
	\end{equation*}
Since $\coh^1(\Gal(H_ p/\Q_p), V^{G_{H_ p}})=\coh^2(\Gal(H_ p/\Q_p), V^{G_{H_ p}})=0$, the restriction to $G_{H_ p}$ gives an isomorphism 	
\begin{equation*}
	\coh^1(\Q_p,V)\longrightarrow \coh^1(H_ p, V)^{\Gal(H_ p/\Q_p)}.
\end{equation*}
Composing it with the identifications
\begin{equation*}
	\coh^1(H_ p, V)^{\Gal(H_ p/\Q_p)}=\coh^1(H_ p, V_f\otimes V_{gh})^{\Gal(H_ p/\Q_p)}=(\coh^1(H_ p, V_f)\otimes V_{gh})^{\Gal(H_ p/\Q_p)},
\end{equation*}
we get the first isomorphism of \eqref{eq: iso cohomology triple product2}.
Finally, the second isomorphism follows from the relation between $\hom$ and tensor and from the selfduality $V_{gh}^\lor\cong V_{gh}$.
\end{proof}

Let $E(H)_L^{V_{gh}}:=\hom_{\Gal(H/\Q)}(V_{gh},E(H)_L)$. The Kummer  homomorphism 
\begin{align*}
	 E(H)_L\longrightarrow\coh^1(H,V_f)
\end{align*} 
induces a homomorphism
\begin{align*}
	\delta: E(H)_L^{V_{gh}}\longrightarrow\hom_{\Gal(H/\Q)}(V_{gh},\coh^1(H,V_f))\cong\coh^1(\Q,V),
\end{align*}
which using \eqref{eq: iso cohomology triple product1} can be seen as an isomorphism 
\begin{align*}
	\delta: E(H)_L^{V_{gh}}\longrightarrow \coh^1(\Q,V).
\end{align*}

For $\triangle, \heartsuit\in\{ \alpha,\beta \}$, denote $$V_{gh}^{\triangle\heartsuit}:=V_g^{\triangle} \otimes V_h^\heartsuit \ \text{ and } \ V^{\triangle\heartsuit}:=V_f\otimes V_g^{\triangle} \otimes V_h^\heartsuit.$$ 
Specializing \eqref{eq: filtration of lambda-adic rep} via $y_g$ we obtain 
\begin{align*}
  0\lra V_g^\beta \lra V_g\lra V_g^\alpha\lra 0.
\end{align*}
For $\triangle\neq\heartsuit$ the pairing \eqref{eq: lambda-adic pairing} and its analog for $h$ induce perfect pairings
 \begin{align}
 	\langle \ , \ \rangle : V_{gh}^{\triangle\triangle} \times V_{gh}^{\heartsuit\heartsuit}\longrightarrow L_p, \qquad 	\langle \ , \ \rangle : V_g^{\triangle\heartsuit} \times V_g^{\heartsuit\triangle}\longrightarrow L_p.\label{eq: pairing and eigenspaces}
 \end{align}
 The identifications
 \begin{align}\label{eq: identifications}
V_{gh}^{\triangle \triangle }\cong \Hom_{L_p[G_{\Q_p}]}(V_{gh}^{\heartsuit\heartsuit},L_p) \ \   \text{ and } \ \ V_{gh}^{\triangle \heartsuit }\cong \Hom_{L_p[G_{\Q_p}]}(V_{gh}^{\heartsuit\triangle},L_p),
 \end{align}
  together with \eqref{eq: iso cohomology triple product2} give the following isomorphisms:

 	\begin{align}
 		\coh^1(\Q_p,V^{\triangle\triangle})\cong(\coh^1(H_p,V_f)\otimes V_{gh}^{\triangle\triangle})^{\Gal(H_p/\Q_p)}\cong\hom_{\Gal(H_p/\Q_p)}(V_{gh}^{\heartsuit\heartsuit},\coh^1(H_p,V_f));\\
 		\coh^1(\Q_p,V^{\triangle\heartsuit})\cong(\coh^1(H_p,V_f)\otimes V_{gh}^{\triangle\heartsuit})^{\Gal(H_p/\Q_p)}\cong\hom_{\Gal(H_p/\Q_p)}(V_{gh}^{\heartsuit\triangle},\coh^1(H_p,V_f)). 		
 	\end{align}

It follows from \cite[Lemma 4.1]{DR3} that the submodule $H^1_f(\Q_p,V_f\otimes V_{gh}^{\triangle\heartsuit})$ and the singular quotient $H^1_s(\Q_p,V_f\otimes V_{gh}^{\triangle\heartsuit})$ can be written in terms of the filtration of $V_f$ as follows:
\begin{align}\label{eq: H1s} 
  \coh^1_s(\Q_p,V_f\otimes V_{gh}^{\triangle\heartsuit})=\coh^1(\Q_p,V_f^-\otimes V_{gh}^{\triangle\heartsuit})\cong(V_{gh}^{\heartsuit\triangle}\otimes\coh^1_s(H_p,V_f))^{\Gal(H_p/\Q_p)};  
\end{align}
\begin{align}\label{eq: H1f}
  \coh^1_f(\Q_p,V_f\otimes V_{gh}^{\triangle\heartsuit}) = \ker(\coh^1(\Q_p,V_f\otimes V_{gh}^{\triangle\heartsuit})\rightarrow\coh^1(I_p,V_f^-\otimes V_{gh}^{\triangle\heartsuit})) = \coh^1(\Q_p,V_f^+\otimes V_{gh}^{\triangle\heartsuit}).
\end{align}

\begin{lemma}\label{lem: kummer maps eigenspaces}
 For $\triangle,\heartsuit\in\{ \alpha,\beta \}$, $\triangle\neq\heartsuit$, there are isomorphisms
	\begin{align*}
& \delta_p: E(H_p)_{L_p}^{V_{gh}}\longrightarrow\coh^1_f(\Q_p,V);\\
& \delta_p^{\triangle\heartsuit}: E(H_p)_{L_p}^{V_{gh}^{\triangle\heartsuit}}\longrightarrow\coh^1_f(\Q_p,V^{\heartsuit\triangle});\\
& \delta_p^{\triangle\triangle}: E(H_p)_{L_p}^{V_{gh}^{\triangle\triangle}}\longrightarrow\coh^1_f(\Q_p,V^{\heartsuit\heartsuit}).
	\end{align*}

\end{lemma}
\begin{proof}
We prove the existence of the isomorphism $\delta^{\triangle\heartsuit}$, the others are similar. By Kummer theory, there is an injective morphism 
	\begin{equation*}
E(H_p)_{L_p}\longrightarrow \coh^1(H_p,V_f),
	\end{equation*}
	which is an isomorphism on its image $\coh^1_f(H_p,V_f)\cong\coh^1(H_p,V_f^+)$.
	It induces an isomorphism 
	\begin{equation*}
\delta_p^{\triangle\heartsuit}: E(H_p)_{L_p}^{V_{gh}^{\triangle\heartsuit}}\longrightarrow \hom_{\Gal(H_p/\Q_p)}(V_{gh}^{\triangle\heartsuit},\coh^1(H_p,V_f^+)).
\end{equation*}	
Using the isomorphisms \eqref{eq: identifications} we obtain
\begin{equation*}
\hom_{\Gal(H_p/\Q_p)}(V_{gh}^{\triangle\heartsuit},\coh^1(H_p,V_f^+))\overset{\cong}{\longrightarrow}(\coh^1(H_p,V_f^+)\otimes V_{gh}^{\heartsuit\triangle})^{\Gal(H_p/\Q_p)}. 
\end{equation*}
Arguing as in the proof of Lemma \ref{lem: first lemma on isomorphisms cohomology triple product}, we get the isomorphisms  
\begin{align*}
	(\coh^1(H_p,V_f^+)\otimes V_{gh}^{\heartsuit\triangle})^{\Gal(H_p/\Q_p)}& \cong 	\coh^1(H_p,V_f^+\otimes V_{gh}^{\heartsuit\triangle})^{\Gal(H_p/\Q_p)}\\
	& \cong \coh^1(\Q_p,V_f^+\otimes V_{gh}^{\heartsuit\triangle})\cong\coh^1_f(\Q_p,V^{\heartsuit\triangle}).
\end{align*}
\end{proof}

From now on we will make the following assumption on the Selmer group of $V$.
\begin{assumption}\label{ass: rank 0}
	$\Sel_p(\Q,V)=0$.
\end{assumption}
Under this assumption one can identify the relaxed Selmer group with the singular quotient.
\begin{lemma}
 Under Assumptions \ref{ass: rank 0} and \ref{ass: p doesnt divide the levels} the natural map
	\begin{equation*}
		\partial_p: \Sel_{(p)}(\Q,V){\longrightarrow} \coh^1_s(\Q_p,V)
	\end{equation*}
is an isomorphism. In particular, there is an isomorphism
\begin{align}\label{eq: relaxed selmer group as singular quotient}
  \Sel_{(p)}(\Q,V)\cong \coh^1_s(\Q_p,V^{\alpha\alpha})\oplus \coh^1_s(\Q_p,V^{\alpha\beta})\oplus \coh^1_s(\Q_p,V^{\beta\alpha})\oplus \coh^1_s(\Q_p,V^{\beta\beta})
\end{align}
\end{lemma}
\begin{proof}
	Since the representation $V$ is self-dual there is an isomorphism $\Sel_{p,*}(V^*)\cong\Sel_{p}(V),$ (see, for example, \cite{BK90} and \cite[Theorem 2.1]{Belaichenotes}). Then the lemma follows immediately from the exact sequence (\ref{eq: exact sequence of Selmer groups}).
\end{proof}
In the next subsection we will describe the spaces in the right hand side of \eqref{eq: relaxed selmer group as singular quotient} in terms of dual exponential maps.

\subsection{Bloch--Kato logarithms and exponentials} The  $G_{\Q_p}$-representations of the form \mbox{$V_f^+\otimes V_{gh}^{\triangle\heartsuit}$} are one dimensional and, therefore, given by characters. Indeed, $G_{\Q_p}$ acts on $V_f^{+}$ as $\chi_{\mathrm{cycl}}\psi_f^{-1}$,  and it acts as $\psi_g$ (resp. $\psi_g^{-1}$) on $V_g^\alpha$ (resp. $V_g^{\beta}$) and as $\psi_h$ (resp. $\psi_h^{-1}$) on $V_h^\alpha$ (resp. $V_h^{\beta}$). Therefore we have that
 \begin{align*}
		V_f^+\otimes V_{gh}^{\alpha\alpha}=L_p(\chi_{\cycl}\psi_f^{-1}\psi_g\psi_h), \ \qquad  \ V_f^+\otimes V_{gh}^{\alpha\beta}=L_p(\chi_{\cycl}\psi_f^{-1}\psi_g\psi_h^{-1}), \\  V_f^+\otimes V_{gh}^{\beta\alpha}=L_p(\chi_{\cycl}\psi_f^{-1}\psi_g^{-1}\psi_h),\ \qquad \ V_f^+\otimes V_{gh}^{\beta\beta}=L_p(\chi_{\cycl}\psi_f^{-1}\psi_g^{-1}\psi_h^{-1}).
	\end{align*}
In particular $V_f^+\otimes V_{gh}^{\triangle\heartsuit}$ is of the form $L_p(\psi\chi_{\cycl})$ for some nontrivial unramified character $\psi$. By \eqref{eq: H1f} we have that $\coh^1_f(\Q_p,V^{\triangle\heartsuit})\cong \coh^1(\Q_p,V_f^+\otimes V_{gh}^{\triangle\heartsuit})$, and the Bloch--Kato logarithm gives an isomorphism (cf. \cite[Example 1.6 (a)]{DR3}):
\begin{equation}\label{eq: BK}
	\log_{\triangle\heartsuit}:\coh^1_f(\Q_p,V^{\triangle\heartsuit})\longrightarrow D(V_f^+\otimes V_{gh}^{\triangle\heartsuit})=D(L_p(\psi\chi_{\cycl})).
      \end{equation}
       For $(\triangle,\heartsuit)=(\alpha,\alpha)$, the pairings \ref{eq: pairing Vg+} and the analogous pairings for $f$ and $h$ give rise to a pairing
\begin{equation}
	\langle \ , \ \rangle:V_f^+\otimes V_g^\alpha\otimes V_h^\alpha \times V_f^-(-1)\otimes V_g^\beta\otimes V_h^\beta \longrightarrow L_p
\end{equation}
which induces 
\begin{equation}\label{eq: pairing D(V)}
\langle \ , \ \rangle:D(V_f^+\otimes V_g^\alpha\otimes V_h^\alpha) \times D(V_f^-(-1)\otimes V_g^\beta\otimes V_h^\beta )\longrightarrow D(L_p)= L_p.
\end{equation}
Denote by $\tilde\omega_f$ the differential form on $X_0(N_f)$ corresponding to $f$. It can be naturally viewed as an element of the de Rham cohomology group $H^1_{\mathrm{dR}}(X_0(N_f)/\Q_p)$. The comparison isomorphisms of $p$-adic Hodge theory provide a natural map
\begin{align*}
 H^1_{\mathrm{dR}}(X_0(N_f)/\Q_p)(1)\lra D(V_f^-) 
\end{align*}
and therefore $\tilde\omega_f$ gives rise to an element $\omega_f\in D(V_f^-(-1))$. In \eqref{eq: pairing D(V)}, pairing with the class $\omega_f\otimes\eta_g\otimes\eta_h$ gives then an isomorphism
\begin{equation}\label{eq: pair with class}
	\langle\cdot, \omega_f\otimes\eta_g\otimes\eta_h\rangle:D(L_p(\psi\chi_{\cycl}))=D(V_f^+\otimes V_g^\alpha\otimes V_h^\alpha)\longrightarrow  L_p.
\end{equation}
There are similar pairings and isomorphisms for the remaining pairs $(\triangle, \heartsuit)$.
We still denote
\begin{equation}\label{eq: block-kato logarithms on V}
\log_{\triangle\heartsuit}:\coh^1_f(\Q_p,V_f\otimes V_{gh}^{\triangle\heartsuit})\longrightarrow  L_p
\end{equation}
the map obtained by composing \eqref{eq: BK} with \eqref{eq: pair with class}.

\begin{remark}\label{rem: logaritmos} The logarithm maps of \eqref{eq: block-kato logarithms on V} are related to the usual $p$-adic logarithm on $E$ as follows. The differential $\omega_f$ gives rise to an invariant differential on $E$, and we denote by 
  \begin{align*}
\log_{f,p}\colon E(H_p)\longrightarrow H_p    
  \end{align*}
the corresponding formal group logarithm on $E$.  The map $\log_{\alpha\beta}$ coincides with the inverse of the isomorphism of Lemma \ref{lem: kummer maps eigenspaces}
\begin{equation*}
E(H_p)_{L_p}^{V_{gh}^{\beta\alpha}}\cong( E(H_p)^{\beta_g\alpha_h}\otimes V_{gh}^{\alpha\beta})^{G_{\Q_p}}
\end{equation*} 
composed with the maps
\begin{align*}
  \begin{array}{ccccc}
	(E(H_p)^{\beta_g\alpha_h}\otimes V_{gh}^{\alpha\beta})^{G_{\Q_p}} & \longrightarrow & (H_p^{\beta_g\alpha_h}\otimes V^{\alpha\beta}_{gh})^{G_{\Q_p}}= D(V^{\alpha\beta}_{gh}) & \longrightarrow & L_p\\
x\otimes v_g^\alpha v_h^\beta & \longmapsto &\log_{f,p}(x)\otimes v_g^\alpha v_h^\beta & & \\
 & & y & \longmapsto & \langle y, \eta_g\omega_h\rangle.
 \end{array}
\end{align*}
Analogous equalities hold for the other maps $\log_{\triangle\heartsuit}$.
\end{remark}

A similar discussion can be applied to the representations of the form $V_f^{-}\otimes V_{gh}^{\triangle\heartsuit}$. In this case we have the following isomorphisms of $1$-dimensional representations:
 \begin{align*}
		V_f^-\otimes V_{gh}^{\alpha\alpha}=L_p(\psi_f\psi_g\psi_h), \ \qquad \ V_f^-\otimes V_{gh}^{\alpha\beta}=L_p(\psi_f\psi_g\psi_h^{-1}\bar\chi), \\ V_f^-\otimes V_{gh}^{\beta\alpha}=L_p(\psi_f\psi_g^{-1}\psi_h\chi),\ \qquad \ V_f^-\otimes V_{gh}^{\beta\beta}=L_p(\psi_f\psi_g^{-1}\psi_h^{-1}).
	\end{align*}
Therefore,  $V_f^-\otimes V_{gh}^{\triangle\heartsuit}$ is isomorphic to a representation of the form $ L_p(\psi)$ for some unramified and nontrivial character  $\psi$. By \eqref{eq: H1s}  there is an identification
\begin{align*}
 \coh^1_s(\Q_p,V^{\triangle\heartsuit})=\coh^1(\Q_p,L_p(\psi)), 
\end{align*}
and by \cite[Example 1.8 (b)]{DR2}  the dual exponential gives isomorphisms
\begin{equation}\label{eq: dual exponentials for V}
\exp^*_{\triangle\heartsuit}:\coh^1_s(\Q_p,V^{\triangle\heartsuit})\longrightarrow D(L_p(\psi))\cong  L_p,
\end{equation}
where the last isomorphism is induced by pairing with the appropriate  class of $D(L_p(\psi^{-1}))=D(V_f^+(-1)\otimes V_g^{\heartsuit}\otimes V_h^\triangle)$ similarly as in \eqref{eq: pair with class}.
Arguing as in Remark \ref{rem: logaritmos}, let 
$$\exp_{f,p}^*:\coh^1_s(H_p,V_f)\longrightarrow H_p$$ denote the dual exponential on $\coh^1_s(H_p,V_f)$.  Then $\exp^*_{\beta\beta}$ can be identified with the composition 
\begin{align}\label{eq: dualexp with pairing}
  \begin{array}{ccccc}
(\coh^1_s(H_p,V_f)^{\alpha_g\alpha_h}\otimes V_{gh}^{\beta\beta})^{G_{\Q_p}} & {\longrightarrow} & (H_p^{\alpha_g\alpha_h}\otimes V^{\beta\beta}_{gh})^{G_{\Q_p}}= D(V^{\beta\beta}_{gh}) & \longrightarrow & L_p\\
x\otimes v_g^\beta v_h^\beta & \longmapsto & \exp_{f,p}^*(x)\otimes v_g^\beta v_h^\beta & & \\
& &  y &\longmapsto & \langle y, \omega_g\omega_h\rangle,
  \end{array}
\end{align}
after taking into account the identification
\begin{align*}
      \coh^1_s(\Q_p,V^{\beta\beta})\cong(\coh^1_s(H_p,V_f)^{\alpha_g\alpha_h}\otimes V_{gh}^{\beta\beta})^{G_{\Q_p}}.
\end{align*}
Analogous formulas hold for the dual exponentials $\exp_{\triangle\heartsuit}$ on the remaining components.





To sum up the discussion of this subsection, we conclude that the relaxed Selmer group of $V$ admits a basis adapted to decomposition \eqref{eq: relaxed selmer group as singular quotient} with respect to the dual exponential maps.
\begin{proposition}
  Under Assumptions \ref{ass: rank 0} and \ref{ass: p doesnt divide the levels}, $\Sel_{(p)}(V)$ has a basis
  \begin{align}\label{eq: basis xi} 
    \{ \xi^{\alpha\alpha}, \xi^{\alpha\beta}, \xi^{\beta\alpha}, \xi^{\beta\beta} \}
  \end{align}
  characterized by the fact that there exist elements $\Psi_{\beta\beta}\in\coh^1_s(H_p,V_f)^{\beta_g\beta_h}, \ \Psi_{\beta\alpha}\in\coh^1_s(H_p,V_f)^{\beta_g\alpha_h}, \ \Psi_{\alpha\beta}\in\coh^1_s(H_p,V_f)^{\alpha_g\beta_h}, \ \Psi_{\alpha\alpha}\in\coh^1_s(H_p,V_f)^{\alpha_g\alpha_h}$ such that  $$\partial_p{\xi}^{\alpha\alpha}=(\Psi_{\beta\beta}\otimes v_g^\alpha v_h^\alpha,0,0,0), \ \partial_p{\xi}^{\alpha\beta}=(0,\Psi_{\beta\alpha}\otimes v_g^\alpha v_h^\beta,0,0)$$
	$$\partial_p{\xi}^{\beta\alpha}=(0,0,\Psi_{\alpha\beta}\otimes v_g^\beta v_h^\alpha,0), \ \partial_p{\xi}^{\beta\beta}=(0,0,0,\Psi_{\alpha\alpha}\otimes v_g^\beta v_h^\beta)$$
	and $$\exp_{f,p}^*(\Psi_{\beta\beta})=\exp_{f,p}^*(\Psi_{\beta\alpha})=\exp_{f,p}^*(\Psi_{\alpha\beta})=\exp_{f,p}^*(\Psi_{\alpha\alpha})=1.$$
\end{proposition}
\begin{remark}
  Notice that the basis \eqref{eq: basis xi}  depends on the choice of the $L$-basis $v_g^\alpha, v_g^\beta$ of $V_g$ and the $L$-basis $v_h^\alpha, v_h^\beta$ of $V_h$. Then each element of the basis $\{ \xi^{\alpha\alpha}, \xi^{\alpha\beta}, \xi^{\beta\alpha}, \xi^{\beta\beta}   \}$ depends on this choice up to multiplication by an element of $L^\times$.
\end{remark}

\section{Special value formula for the triple product $p$-adic $\L$-function in rank $0$}\label{sec: main thm}
We continue with the notation and assumptions of the previous section. In particular, $V:=V_{f}\otimes V_g\otimes V_h$ is the tensor product of the $p$-adic representations attached to the newforms $$f\in S_2(N_f)_\Q, \  \ g\in M_1(N_g,\chi)_L, \ \ h\in M_1(N_h,\bar{\chi})_L,$$
and we assume from now on that $\gcd(N_f, N_g, N_h)$ is square free. Recall that $V_g^{\circ}$ (resp. $V_h^{\circ}$) stands for the Artin representation attached to $g$ (resp. $h$) and $\rho$ denotes the tensor product representation
\begin{align*}
  \rho: \Gal(H/\Q)\longrightarrow\GL(V^{\circ}_{g}\otimes V_h^{\circ})\cong\GL_4(L).
\end{align*}
The complex $L$-function $$L(E, \rho,s)=L(f\otimes g\otimes h,s)$$ has entire continuation and satisfies a functional equation relating the value at $s$ with the value at $2-s$.
Let $\epsilon$ be the  sign of this functional equation and denote $N:=\mathrm{lcm}(N_f,N_g,N_h)$. Then $\epsilon$ is the product of local signs
\begin{align*}
  \epsilon=\prod_{v}\epsilon_v,
\end{align*}
where $v$ runs over the places of $\Q$ dividing $N$ or $\infty$. In this setting, $\epsilon_\infty=+1$. Assume also that $\epsilon_v=+1$ for all $v\mid N$. In particular, the global sign is $\epsilon = 1$ and the order of vanishing of $L(E,\rho,s)$ at the central point $s=1$ is even.

Recall that $p$ stands for a prime that does not divide $N$, and that $\bfg\in \Lambda_\bfg[[q]]$ (resp. $\bfh\in\Lambda_\bfh[[q]]$) is a Hida family passing through the $p$-stabilization $g_\alpha$ (resp. $h_\alpha$) such that $U_pg_\alpha=\alpha_g g_\alpha$ (resp. $U_ph_\alpha=\alpha_h h_\alpha$). Similarly, denote by $\bff\in \Lambda_\bff[[q]]$ a Hida family passing through the $p$-stabilization $f_\alpha$ of $f$. 

Denote by $\L_p^g(\bff,\bfg,\bfh)$ the triple product $p$-adic $\L$-function defined in \cite{DR2}, attached to the choice of $\Lambda$-adic test vector $(\breve{\bff},\breve{\bfg},\breve{\bfh})$ of \cite[Chap. 3]{Hsieh}. The values $\L_p^g(\bff,\bfg,\bfh)(k,\ell,m)$ of this $p$-adic $\L$-function at triples of integers $(k,\ell,m)$ with $\ell\geq k+m$ interpolate the square root of the algebraic part of
\begin{align}\label{eq: interpolation}
  L(\breve{\bff}_k\otimes \breve{\bfg}_\ell\otimes \breve{\bfh}_m,\frac{k+\ell+m-2}{2}),
\end{align}
where $\breve{\bff}_k,\breve{\bfg}_\ell, \breve{\bfh}_m$ denote the specializations of $\breve{\bff}, \breve{\bfg}, \breve{\bfh}$ at weights $k,\ell,m$.

There is an analogous triple product $p$-adic $\L$-function $\L_p^f(\bff,\bfg,\bfh)$ that interpolates \eqref{eq: interpolation} but for the values $(k,\ell,m)$ with $k\geq \ell + m$. In particular, $\L_p^f(\bff,\bfg,\bfh)(2,1,1)$ is directly related to $L(E,\rho,1)$.

The article \cite{DLR} studies the value $\L_p^g(\bff,\bfg,\bfh)(2,1,1)$ when $L(E,\rho,1)\neq 0$. In particular, the Elliptic Stark Conjecture predicts that when $E(H)^\rho$ is $2$-dimensional $\L_p^g(\bff,\bfg,\bfh)(2,1,1)$ is related to the $p$-adic logarithms of global elements in $E(H)^\rho$. 

In the present note, our running Assumption \ref{ass: rank 0} is that $\Sel_p(\Q,V)=0$. This implies that $E(H)^\rho=0$ and, conjecturally, it also implies that $L(E,\rho,1)\neq 0$. The main result of this section is an explicit formula for $\L_p^g(\bff,\bfg,\bfh)(2,1,1)$ in this case, and this can be seen as completing the study of $\L_p^g(\bff,\bfg,\bfh)(2,1,1)$ initiated in \cite{DLR}. The main tool that we will use are the generalized Kato classes 
\begin{align*}
  \kappa:=\kappa(f,g_\alpha,h_\alpha)\in\Sel_{(p)}(\Q,V)
\end{align*}
introduced in \cite[\S3]{DR2}. Next, we briefly recall the relation between $\kappa$ and the  $p$-adic $L$-values $\L^f_p(\bff,\bfg,\bfh)(2,1,1)$ and $\L^g_p(\bff,\bfg,\bfh)(2,1,1)$. To lighten the notation, from now on we put $\L^f_p:=\L^f_p(\bff,\bfg,\bfh)(2,1,1)$ and $\L_p^g:=\L^g_p(\bff,\bfg,\bfh)(2,1,1)$. Let also
\begin{align*}
  \pi_{\alpha\beta}:\coh^1(\Q_p,V)\longrightarrow \coh^1(\Q_p,V^{\alpha\beta})
\end{align*}
be the projection map induced by the natural map $V\ra V^{\alpha\beta}$. 


\begin{proposition}[Darmon--Rotger]\label{prop: relation with Lpf}
  \begin{enumerate}
  \item The element $\partial_p\kappa$ lies in the image of the natural map
  \begin{align*}
   \coh^1_s(\Q_p,V^{\beta\beta})\lra \coh_s^1(\Q_p,V) .
  \end{align*}
In particular, $\partial_p\kappa$ can be viewed as an element of  $\coh^1_s(\Q_p,V^{\beta\beta})$. Moreover,
  	\begin{equation}\label{eq: relation with Lpf}
	\exp_{\beta\beta}^*(\partial_p\kappa)=  \dfrac{2(1-p\alpha_f\alpha_g^{-1}\alpha_h^{-1})}{\alpha_g\alpha_h(1-\alpha_f^{-1}\alpha_g\alpha_h)(1-\chi^{-1}(p)\alpha_f^{-1}\alpha_g\alpha_h^{-1})}\times \L_p^f.
      \end{equation}

      \item The element $\pi_{\alpha\beta}\res_p\kappa\in\coh^1(\Q_p,V^{\alpha\beta})$ belongs to $\coh_f^1(\Q_p, V^{\alpha\beta})$, and 
	\begin{equation}\label{eq: relation with Lpg}
\log_{\alpha\beta}(\pi_{\alpha\beta}\res_p\kappa)=\L_p^g \times 2(1-\chi(p)p^{-1}\alpha_fa_p(g)^{-1}a_p(h))^{-1}.
\end{equation}

  \end{enumerate}
  \end{proposition}
  \begin{proof} The fact that $\partial_p\kappa$ is the image of an element in $\coh^1_s(\Q_p,V^{\beta\beta})$ is \cite[Proposition 2.8]{DR2}. The equality \eqref{eq: relation with Lpf} follows from Proposition 5.2 and  Theorem 5.3 of \cite{DR2}. 
  By part $(1)$ of the proposition  $\pi^s_{\alpha\beta}\partial_p\kappa=0$ in the singular quotient $\coh_s^1(\Q_p, V^{\alpha\beta})$. This means that $\pi_{\alpha\beta}\res_p\kappa$ belongs to $\coh^1_f(\Q_p,V^{\alpha\beta})$. Equality \eqref{eq: relation with Lpg} follows from Proposition 5.1, Theorem 5.3 of \cite{DR2}. 
\end{proof}

Using the $\kappa$ recalled above and the basis \eqref{eq: basis xi} of $\Sel_{(p)}(V)$, we can give a precise formula for $\L_p^g$ in the rank $0$ setting. Define $R_{\beta\alpha}\in E(H_p)^{\beta\alpha}$ by the equality
\begin{equation}\label{eq: def of Ralphabeta}
	\pi_{\alpha\beta}\res_p{\xi}^{\beta\beta}={R}_{\beta\alpha}\otimes v_g^{\alpha}v_h^{\beta}\in\coh_f^1(\Q_p,V^{\alpha\beta})=(E(H_p)^{\beta\alpha}\otimes V_{gh}^{\alpha\beta})^{G_{\Q_p}}.
\end{equation}

\begin{theorem}\label{thm: main}
The class $\kappa$ is a multiple of $\xi^{\beta\beta}$; more precisely,
\begin{align*}
  \kappa=\dfrac{\Theta_g\Theta_h2(1-p\alpha_f\alpha_g^{-1}\alpha_h^{-1})\L_p^f }{\alpha_g\alpha_h(1-\alpha_f^{-1}\alpha_g\alpha_h)(1-\chi^{-1}(p)\alpha_f^{-1}a_p(g)a_p(h)^{-1})}\cdot{\xi}^{\beta\beta}.
\end{align*}
Moreover, if we define the quantities
\begin{align*}
  \L_{g_\alpha}:=\dfrac{\Omega_g}{\Theta_g}, \ \ \ \cE:=\dfrac{(1-\chi(p)p^{-1}\alpha_g^{-1}\alpha_h)(1-p\alpha_f\alpha_g^{-1}\alpha_h^{-1})}{\alpha_g\alpha_h(1-\alpha_f^{-1}\alpha_g\alpha_h)(1-\chi^{-1}(p)\alpha_f^{-1}\alpha_g\alpha_h^{-1})}
\end{align*}
then we have that
\begin{align*}
\L^g_p={\cE }\times\dfrac{\log_p({R}_{\beta\alpha})}{\L_{g_\alpha}}\times{\L_p^f} \mod L^\times.
\end{align*}
\end{theorem}
\begin{proof}
	By Proposition \ref{prop: relation with Lpf}, $\kappa$ is an element of $\Sel_{(p)}(\Q,V)$ such that
	\begin{equation}\label{eq: exp kappa}
		\exp^*(\partial_p\kappa)=(0,0,0,\dfrac{2(1-p\alpha_f\alpha_g^{-1}\alpha_h^{-1})}{\alpha_g\alpha_h(1-\alpha_f^{-1}\alpha_g\alpha_h)(1-\chi^{-1}(p)\alpha_f^{-1}\alpha_g\alpha_h^{-1})}\times \L_p^f).
	\end{equation} 
	Then $\kappa$ is a multiple of the element $\xi^{\beta\beta}$; indeed
	$$\kappa=\dfrac{\exp^*_{\beta\beta}(\partial_p\kappa)}{\exp^*_{\beta\beta}(\partial_p{\xi}^{\beta\beta})}{\xi}^{\beta\beta}.$$
Observe that \eqref{eq: exp kappa} gives us the expression for the numerator. We now compute the denominator. 
	\begin{align*}
		\exp^*_{\beta\beta}(\partial_p{\xi}^{\beta\beta})& = \langle  \exp_{f,p}^* (\Psi_{\alpha\alpha})\otimes v_g^\beta v_h^\beta,\omega_g\omega_h \rangle\\
		 &=\dfrac{\exp_{f,p}^* (\Psi_{\alpha\alpha})}{\Theta_g\Theta_h}\\
		 & = \dfrac{1}{\Theta_g\Theta_h}.
	\end{align*}
	Here we used the fact that $\eta_g\eta_h=\Theta_g\Theta_h v_g^\beta v_h^\beta$.
	So we get 
	\begin{align*}
	\kappa & =\dfrac{\exp^*_{\beta\beta}(\partial_p\kappa)}{\exp^*_{\beta\beta}(\partial_p{\xi}^{\beta\beta})}\cdot{\xi}^{\beta\beta}\\
	& = \dfrac{2(1-p\alpha_f\alpha_g^{-1}\alpha_h^{-1})\Theta_g\Theta_h}{\alpha_g\alpha_h(1-\alpha_f^{-1}\alpha_g\alpha_h)(1-\chi^{-1}(p)\alpha_f^{-1}\alpha_g\alpha_h^{-1})}\times \L_p^f\cdot{\xi}^{\beta\beta}
	\end{align*} 
	By \eqref{eq: relation with Lpg},
	\begin{align*}
		\L_p^g & =\frac{1}{2}(1-\chi(p)p^{-1}\alpha_g^{-1}\alpha_h)\log_{\alpha\beta}(\pi_{\alpha\beta}\res_p\kappa)\\
		& = \dfrac{\Theta_g\Theta_h(1-\chi(p)p^{-1}\alpha_g^{-1}\alpha_h)(1-p\alpha_f\alpha_g^{-1}\alpha_h^{-1})}{\alpha_g\alpha_h(1-\alpha_f^{-1}\alpha_g\alpha_h)(1-\chi^{-1}(p)\alpha_f^{-1}\alpha_g\alpha_h^{-1})}\times \L_p^f\log_{\alpha\beta}(\pi_{\alpha\beta}\res_p{\xi}^{\beta\beta})\\
		& =\cE \Theta_g\Theta_h \L_p^f \log_{\alpha\beta}(\pi_{\alpha\beta}\res_p{\xi}^{\beta\beta})\\
		& = \cE \Theta_g\Theta_h \L_p^f\langle \log_p({R}_{\beta\alpha})\otimes v_g^\alpha v_h^\beta,\eta_g\omega_h  \rangle\\
		& = {\cE\L_p^f }\dfrac{  \Theta_g\Theta_h }{\Omega_g\Theta_h}\log_p({R}_{\beta\alpha})\\
		& = \cE\L_p^f\dfrac{\Theta_g}{\Omega_g}\log_p({R}_{\beta\alpha})\\
		& = \dfrac{\cE\L_p^f }{\L_{g_\alpha}}\log_p({R}_{\beta\alpha})
	\end{align*}
	since $\omega_g\eta_h=\Omega_g\Theta_h\otimes v_g^\alpha v_h^\beta$.
\end{proof}

We end this section by noting that $\L_{g_\alpha}$ is often expected to be related to the Gross--Stark Unit $u_{g_\alpha}$ attached to the modular form $g_\alpha$ as defined in \cite[\S1]{DLR}. More precisely, under the additional assumption that $g$ is not the theta series of a Hecke character of a real quadratic field in which $p$ splits, \cite[Conjecture 2.1]{DR2b} predicts that
\begin{align}\label{eq: conj unit}
  \L_{g_\alpha}\stackrel{?}{=}\log_p(u_{g_\alpha}) \ \mod L^\times.
\end{align}
Thus we obtain the following consequence of Theorem \ref{thm: main}, under the aforementioned hypothesis:
\begin{corollary}
	Assuming the equality \eqref{eq: conj unit}, if $\Sel_p(\Q,V)=0$ then 
	$$\L^g_p=\mathcal{E}\times\dfrac{\log_p(R_{\beta\alpha})}{\log_p(u_{g_\alpha})}\times \L_p^f.$$
\end{corollary}

\section{The case of theta series of an imaginary quadratic field $K$ where $p$ splits }\label{sec: p split in K}
In this section we will consider a particular case where $g$ and $h$ are theta series of the same imaginary quadratic field in which $p$ splits. We will see that in this setting the representation $V$ decomposes in a way that forces $\L_p^g$ to vanish when the complex $L$-function does not vanish at the central critical point; that is, the special value of the $p$-adic $\L$-function vanishes in analytic rank $0$.

Let $K$ be an imaginary quadratic field of discriminant $D_K$. Let $\psi_g,\psi_h\colon \A_K^\times\ra \C^\times$ be two finite order Hecke characters of $K$ of conductors $\fc_g,\fc_h$ and central characters $\varepsilon, \bar{\varepsilon}$ respectively. Here $\varepsilon \colon \A_\Q^\times\ra \C^\times$  is a finite order character of  and $\bar{\varepsilon}$ denotes is complex conjugate. Let $g$ and $h$ be the theta series attached to $\psi_g$ and $\psi_h$. They are modular forms of weight one, and their levels and nebentype characters are given by
\begin{align*}
N_g:=D_K\cdot\norm_K(\fc_g), \  \ N_h:=D_K\cdot\norm_K(\fc_h), \ \ \chi:=\chi_K\cdot\varepsilon, \ \ \bar{\chi}=\chi_K\cdot\bar{\varepsilon},
\end{align*}
where $\norm_K$ stands for the norm on ideals of $K$ and we regard $\varepsilon$ and $\bar{\varepsilon}$ as Dirichlet characters via class field theory. That is to say,
\begin{align*}
g\in M_1(N_g,\chi), \ \ \text{and} \ \  h\in M_1(N_h,\bar{\chi}).
\end{align*}
Let $f\in S_2(N_f)$ be a newform with rational coefficients and let $E$ be the associated elliptic curve over $\Q$. We will particularize some of the results of the previous sections to this choice of forms $f$, $g$, and $h$, so we will use the same notations as before. In particular, $\rho$ stands for the Artin representation afforded by $V_g\otimes V_h$ and $p$ is a prime that does not divide $N_f\cdot N_g\cdot N_h$. In this section, we will make the following additional assumptions:
\begin{enumerate}
\item $\gcd(N_f,\fc_g\fc_h)=1;$
\item $p$ splits in $K$.
\end{enumerate}

A finite order Hecke character $\psi$ of $K$ can be regarded, via class field theory, as a Galois character $\psi\colon G_K\ra\A_K^\times$. Let $\sigma_0$ be any element in  $G_\Q\setminus G_K$. We denote by $\psi'$ the character defined by $\psi'(\sigma):=\psi(\sigma_0\sigma\sigma_0^{-1})$ (this does not depend on the particular choice of $\sigma_0$). Also, $\psi$ gives rise to a $1$-dimensional representation of $G_K$, and we let  $V_\psi=\Ind_K^\Q(\psi)$ denote the induced representation; it is a $2$-dimensional representation of $G_\Q$. Observe that, with this notation, $V_g=V_{\psi_g}$ and $V_h=V_{\psi_h}$.

There is a well-known decomposition of $V_g\otimes V_h$ as the direct sum of two representations:  
\begin{equation}\label{eq: factorisation of reps attached to gh}
	V_g\otimes V_h=V_{\psi_1}\oplus V_{\psi_2},
\end{equation}
where the characters $\psi_1$ and $\psi_2$ are
\begin{align*}
  \psi_1:=\psi_g\psi_h,  \ \ \text{and} \ \ \psi_2:=\psi_g\psi'_h.
\end{align*}
This induces a decomposition of the representation $V=V_f\otimes V_g\otimes V_h$ as a direct sum of two representations:
\begin{equation}\label{eq: factorisation of reps}
V=V_1\oplus V_2,
\end{equation}
where
\begin{align*}
V_1:=V_f\otimes V_{\psi_1}, \ \ \text{and} \ \  V_2:=V_f\otimes V_{\psi_2}.  
\end{align*}
This induces a factorization of complex $L$-functions
\begin{align*}
  L(E,\rho,s)=L(E,\psi_1,s)\cdot L(E,\psi_2,s).
\end{align*}
Under our assumption that $\gcd(N_f,\fc_g\fc_h)=1$ the local signs of $L(E,\psi_1,s)$ and $L(E,\psi_2,s)$ are equal, so that the local signs of $L(E,\rho,s)$ are all equal to $+1$ and therefore the assumption on local signs of Section \ref{sec: main thm} is satisfied.

\begin{theorem}\label{thm: CM case} In the setting of this section, if   $L(E,\rho,1)\neq 0$ then $\L_p^g=0.$
\end{theorem}
\begin{proof}
If $L(E,\rho,1)\neq 0$ then $L(E,\psi_i,1)\neq 0$ for $i=1,2$. Note that $\psi_1$ and $\psi_2$ are ring  class characters of the imaginary quadratic field $K$. Then, by results of Gross--Zagier and Kolyvagin
\begin{equation}\label{eq: algebraic rank (0,0)}
	\Sel_p(\Q,V_i)=0 \ \ \text{for } i=1,2.
\end{equation}
The decomposition \eqref{eq: factorisation of reps} induces a decomposition of the Selmer groups 
\begin{equation}\label{eq: decomposition of Selmer}
	\Sel_p(\Q,V)=\Sel_p(\Q,V_1)\oplus\Sel_p(\Q,V_2),
\end{equation} 
 and analogously for the relaxed and the strict Selmer groups of $V$. In particular, $\Sel_p(\Q,V)=0.$

Since $p$ splits in $K$ we can write $p\cO_K=\p\bar{\p}$, and from our assumption that $p\nmid N_f\cdot N_g\cdot N_h$ we see that $p\nmid \fc_g\fc_h$. Without loss of generality we can suppose that 
$$\psi_g(\fp)=\alpha_g, \  \ \psi_g(\bar{\p})=\beta_g, \ \ \psi_h(\fp)=\alpha_h, \  \ \psi_h(\bar{\p})=\beta_h,$$ so that 
$$V_1=V^{\alpha\alpha}\oplus V^{\beta\beta} \ \ \text{and} \ \ V_2=V^{\alpha\beta}\oplus V^{\beta\alpha}.$$ 

By (\ref{eq: algebraic rank (0,0)}), the same computations as in \S\ref{eq: selmer structure} show that there are isomorphisms
\begin{align*}
  \Sel_{(p)}(\Q,V_1)\overset{\partial_p}{\longrightarrow}\coh^1_s(\Q_p,V_1)\overset{(\pi^s_{\alpha\alpha},\pi^s_{\beta\beta})}{\longrightarrow}\coh^1_s(\Q_p,V_1^{\alpha\alpha})\oplus\coh^1_s(\Q_p,V_1^{\beta\beta}),
\end{align*}
where $\pi_{\alpha\alpha}^s$ denotes the natural map in the singular quotient induced by the projection $V\ra V^{\alpha\alpha}$, and analogously for $\pi_{\beta\beta}^s$.
Similarly, there are dual exponential maps
\begin{align*}
  \exp^*_{\alpha\alpha}:\coh^1_s(\Q_p,V_1^{\alpha\alpha})=\coh^1(\Q_p,V_f^-\otimes V_{gh}^{\alpha\alpha}){\longrightarrow} L_p
\end{align*}
and 
\begin{align*}
\exp^*_{\beta\beta}:\coh^1_s(\Q,V_1^{\beta\beta})=\coh^1(\Q_p,V_f^-\otimes V_{gh}^{\beta\beta}){\longrightarrow} L_p
\end{align*}
which are in fact isomorphisms.

Then $\Sel_{(p)}(\Q,V_1)$ has dimension $2$ over $\Q_p$ with the canonical basis
\begin{align*}
  \zeta^{\alpha\alpha}, \ \zeta^{\beta\beta},
\end{align*}
where $\zeta^{\alpha\alpha}$ is characterized (up to scalars in $L^\times$) by the fact that
$$\exp^*_{\alpha\alpha}(\pi_{\alpha\alpha}\partial_p(\zeta^{\alpha\alpha}))=1, \ \ \text{and} \ \ \exp^*_{\beta\beta}(\pi_{\beta\beta}\partial_p(\zeta^{\alpha\alpha}))=0.$$
Similarly, 
$$\exp^*_{\alpha\alpha}(\pi_{\alpha\alpha}\partial_p(\zeta^{\beta\beta}))=0, \ \ \text{and} \ \ \exp^*_{\beta\beta}(\pi_{\beta\beta}\partial_p(\zeta^{\beta\beta}))=1.$$
Analogously, $\Sel_{(p)}(\Q,V_2)$ has dimension $2$ with basis $\zeta^{\alpha\beta}, \zeta^{\beta\alpha}$.

By Theorem \ref{thm: main}, the value $\L_p^g$ is a multiple of
$\log_{\alpha\beta}(\res_p\xi^{\beta\beta})$. On the other hand,  using the decomposition $$\Sel_{(p)}(\Q,V)=\Sel_{(p)}(\Q,V_1)\oplus\Sel_{(p)}(\Q,V_2),$$ the element $\xi^{\beta\beta}\in\Sel_{(p)}(\Q,V)$  corresponds to a multiple of $(0,\zeta^{\beta\beta})$, and this implies that $$\pi_{\beta\alpha}\res_p\xi^{\beta\beta}=0.$$
\end{proof}

\section{Numerical computations}\label{sec: Numerical computations}



In this section we present some numerical examples. They have been computed with a Sage (\cite{Sage}) implementation of Lauder's algorithms (\cite{Lauder}), adapted to work in the current setting. The code is available at \url{github.com/mmasdeu/ellipticstarkconjecture}. The data for the weight-one modular forms can be found in Alan Lauder's website.\footnote{See \url{http://people.maths.ox.ac.uk/lauder/weight1/}.}

\subsection{Dihedral case}
\begin{enumerate}[$(a)$]
\item We computed $\L_p^g(\bff,\bfg,\bfg)(2,1,1)$ with $\bff$ the Hida family passing through the modular form $f_E$ of weight $2$ attached to an elliptic curve $E/\Q$ of conductor $N_f$ and $\bfg$ attached to the weight-one modular form  $g=\theta(1_K)$ for some imaginary quadratic field $K$. The modular form $g$ belongs then to $M_1(N_g,\chi_K)_\Q$. For each of the entries in the table we give the Cremona label for the elliptic curve $E_f$, its conductor $N_f$, the field $K$, the level $N_g$ of $g$, the level $N$ such that $p^\alpha N=\operatorname{lcm}(N_f,N_g)$ with $\alpha\geq 0$ and  $p\nmid N$. In all of these cases, we obtained $\L_p^g(\bff,\bfg,\bfg)=0$ up to the working precision of $p^{10}$.   Due to computational restrictions, only in the ramified case we have been able to compute examples where $p$ divides the conductor of the elliptic curve.

  Note that all the elliptic curves arising in Table 1 below have rank $0$ over $K$, and thus the zeros obtained in this table are accounted for by Theorem~\ref{thm: CM case}.

\begin{table}[h!]
  \label{table:split}
  \begin{tabular}{lllllc}
  \toprule
$E_f$& $K$ & $N_g$ & $p$ & $N$ & $\L_p^g(\bff,\bfg,\bfg)$\\
  \midrule
11a & $\Q(\sqrt{-5})$&$20$ & $7$ & $220$ & 0\\
11a  & $\Q(\sqrt{-11})$&$11$ & $5$ & $11$& 0\\
19a & $\Q(\sqrt{-19})$&$19$ & $5$ & $19$& 0\\
19a  & $\Q(\sqrt{-19})$&$19$ & $7$ & $19$& 0\\
39a  & $\Q(\sqrt{-39})$&$39$ & $5$ & $39$& 0\\
51a  & $\Q(\sqrt{-51})$&$51$ & $5$ & $51$& 0\\
55a & $\Q(\sqrt{-55})$&$55$ & $7$ & $55$& 0\\
187a  & $\Q(\sqrt{-187})$&$187$ & $7$ & $187$& 0\\
\bottomrule\\
  \end{tabular}
    \caption{Cases with $p$ split in $K$.}
  \end{table}

In the next two tables we see instances of zeros which we expect are explained by the sign of the action of the level $N$ Atkin-Lehner operator although we have not verified this in detail.

\begin{table}[h!]
\begin{tabular}{lllllc}
\toprule
$E_f$ & $K$ & $N_g$ & $p$ & $N$ & $\L_p^g(\bff,\bfg,\bfg)$\\
  \midrule
11a  & $\Q(\sqrt{-3})$&$3$ & $5$ & $33$& 0\\
11a & $\Q(\sqrt{-11})$&$11$ & $7$ & $11$& 0\\
15a  & $\Q(\sqrt{-15})$&$15$ & $7$ & $15$& 0\\
39a  & $\Q(\sqrt{-39})$&$39$ & $7$ & $39$& 0\\
51a  & $\Q(\sqrt{-51})$&$51$ & $7$ & $51$& 0\\
67a  & $\Q(\sqrt{-67})$&$67$ & $5$ & $67$& 0\\
67a  & $\Q(\sqrt{-67})$&$67$ & $7$ & $67$& 0\\
187a  & $\Q(\sqrt{-187})$&$187$ & $5$ & $187$& 0\\

\bottomrule\\
\end{tabular}
\caption{Cases with $p$ inert in $K$.}
\end{table}

\begin{table}[h!]
  \begin{tabular}{lllllc}
    \toprule
$E_f$ & $K$ & $N_g$ & $p$ & $N$ & $\L_p^g(\bff,\bfg,\bfg)$\\
  \midrule
15a  & $\Q(\sqrt{-15})$&$15$ & $5$ & $3$& 0\\
35a  & $\Q(\sqrt{-35})$&$35$ & $5$ & $7$& 0\\
35a  & $\Q(\sqrt{-35})$&$35$ & $7$ & $5$& 0\\
55a  & $\Q(\sqrt{-55})$&$55$ & $5$ & $11$& 0\\

\bottomrule\\
  \end{tabular}

  \caption{Cases with $p$ ramified in $K$.}
\end{table}

In what follows we illustrate with examples the fact that the quantity $\L_p^g(\bff,\bfg,\bfg)$ is not always zero. 

\item In this example we fix $f$ to be attached to the elliptic curve $E_f\colon y^2=x^3+x^2-15x+18$, of conductor $N_f=120$. The weight-one form $g$ we consider has level $N_g=120$ also, and has $q$-expansion
  \begin{align*}
    g(q)&=    q + iq^{2} + iq^{3} - q^{4} -iq^{5} - q^{6} -iq^{8} - q^{9} + q^{10} -iq^{12} + q^{15} + q^{16} -iq^{18}\\
    &+ iq^{20} + q^{24} - q^{25} -iq^{27} + iq^{30} - 2q^{31} + iq^{32} + O(q^{34}),
  \end{align*}
  where $i^2=-1$. It is the theta series attached to the Dirichlet character $\epsilon$ modulo $120$ defined by
  \[
    \epsilon(97)=-1, \quad \epsilon(31)=1,\quad \epsilon(41)=-1,\quad\epsilon(61)=-1.
  \]
  The field cut out by $\epsilon$ is $K=\Q(\sqrt{-6})$, and we take $p=5$ which is split in both $L=\Q(\sqrt{-1})$ and $K$. Note that $p$ divides $N_f$ and $N_g$. We compute to precision $10$ the quantity
  \[
\L_5^g(\bff,\bfg,\bfg)(2,1,1)=   4 \cdot 5 + 3 \cdot 5^{2} + 4 \cdot 5^{3} + 3 \cdot 5^{5} + 4 \cdot 5^{6} + 3 \cdot 5^{7} + 5^{8} + 2 \cdot 5^{9} + O(5^{10}) .
    \]

    With the same setting, we take $p=13$ (now $p$ is split in $L$ but inert in $K$). We obtain
    \[
\L_{13}^g(\bff,\bfg,\bfg)(2,1,1) = 7 + 3 \cdot 13 + 10 \cdot 13^{2} + 13^{4} + 11 \cdot 13^{5} + 13^{6} + 6 \cdot 13^{7} + 4 \cdot 13^{8} + 5 \cdot 13^{9} + O(13^{10})
      \]

      
    \item Let $E_f$ be the elliptic curve $y^2+y = x^3+x^2+42x -131$ with label \texttt{175c1}. It has conductor $N_f=175$ and rank $0$. Let $g=h$ be the theta series of the character $\epsilon_1$ of $K=\Q(\alpha)$ with $\alpha$ satisfying $\alpha^2-\alpha+2=0$, of discriminant $D_K=-7$ and conductor $5\cO_K$ (which is inert, of norm $25$), satisfying
      \[
        \epsilon_1(127) = -1,\quad \epsilon_1(101)=-1.
      \]
The modular form $g$ has $q$-expansion
      \[
g(q) = q + iq^{2} - iq^{7} + iq^{8} - q^{9} - q^{11} + q^{14} - q^{16} - iq^{18} - iq^{22} - iq^{23} + q^{29} + O(q^{30}),
\]
where again $i^2=1$. For $p=13$ (which is inert in $K$ and split in $L$), we obtain
\[
\L_{13}^g(\bff,\bfg,\bfg)(2,1,1)= 1 + 3 \cdot 13 + 2 \cdot 13^{2} + 13^{3} + 12 \cdot 13^{4} + 9 \cdot 13^{5} + 3 \cdot 13^{8} + 5 \cdot 13^{9} + O(13^{10}).
\]

	\item Finally, consider the elliptic curve $E_f$ of conductor $175$ from the previous example, and for $g=h$ consider the theta series of another character $\epsilon_2$ of $K=\Q(\alpha)$, $\alpha^2-\alpha+2=0$,  of discriminant $D_K=-7$ and conductor $5\cO_K$ (inert, of norm $25$), now taking the values
          \[
            \epsilon_2(127) = 1,\quad \epsilon_2(101) = -1.
      \]
This yields a modular form $g$ with $q$-expansion
      \[
        g(q) = q + q^{2} - q^{7} - q^{8} + q^{9} - q^{11} - q^{14} - q^{16} + q^{18} - q^{22} + q^{23} - q^{29} + O(q^{30}).
      \]
     We numerically obtain for $p=13$ that
      \[
\L_{13}^g(\bff,\bfg,\bfg)(2,1,1) = 0.
\]
Again, we do not have a way to prove that $\L_{13}^g(\bff,\bfg,\bfg)(2,1,1)$ is actually zero.

\end{enumerate}

\subsection{Exotic image case}
In the non-CM setting, we have been able to compute the following example. Consider $E_f \colon y^2=x^3-17x-27$, which has conductor $N_f=124$. Let $g$ be the modular form of level $N_g=124$ and projective image $A_4$, defined as the theta series of the character $\epsilon$ of conductor $124$ having values
\[
  \epsilon(65) = \alpha^2-1,\quad \epsilon(63)=-1,
\]
where $\alpha$ satisfies $\alpha^4-\alpha^2+1=0$. The modular form $g$ has $q$-expansion
\begin{align*}
  g(q)&= q - \alpha^{3}q^{2} + \left(-\alpha^{3} + \alpha\right)q^{3} - q^{4} + \left(\alpha^{2} - 1\right)q^{5} - \alpha^{2}q^{6} + \left(\alpha^{3} - \alpha\right)q^{7} + \alpha^{3}q^{8} + \alpha q^{10} - \alpha q^{11}\\
      & + \left(\alpha^{3} - \alpha\right)q^{12}+ \left(-\alpha^{2} + 1\right)q^{13} + \alpha^{2}q^{14} + \alpha^{3}q^{15} + q^{16} - \alpha^{2}q^{17} + \left(-\alpha^{3} + \alpha\right)q^{19} + \left(-\alpha^{2} + 1\right)q^{20}\\
  &+ \left(\alpha^{2} - 1\right)q^{21} + \left(\alpha^{2} - 1\right)q^{22} + \alpha^{2}q^{24} - \alpha q^{26} + \alpha^{3}q^{27} + O(q^{28}).
\end{align*}
We let $h=g^*$ its complex conjugate, and compute with $p=13$, obtaining
\[
  \L_{13}^g(\bff,\bfg,\bfh)(2,1,1)= 1 + 5 \cdot 13 + 5 \cdot 13^{2} + 4 \cdot 13^{3} + 6 \cdot 13^{4} + 6 \cdot 13^{5} + 6 \cdot 13^{6} + 13^{7} + 3 \cdot 13^{8} + 9 \cdot 13^{9} + 9 \cdot 13^{10} + O(13^{11}).
\]



\bibliographystyle{alpha}
\bibliography{refs2}

\end{document}